\documentclass[preprint,12pt]{elsarticle}

\usepackage[T1]{fontenc}
\usepackage[utf8]{inputenc}
\usepackage{lmodern}
\usepackage{amsmath,amssymb,amsthm,mathtools}
\usepackage{microtype}
\usepackage[a4paper,margin=24mm,footskip=10mm]{geometry}
\usepackage{etoolbox}
\usepackage{enumitem}
\usepackage{needspace}
\usepackage{xcolor}
\usepackage{hyperref}

\biboptions{numbers,sort&compress}
\hypersetup{
  colorlinks=true,
  linkcolor=blue!45!black,
  citecolor=blue!45!black,
  urlcolor=blue!45!black,
  pdftitle={A non-sober core-compact meet-continuous dcpo},
  pdfsubject={Scott topology, meet-continuity, core-compactness, local compactness and sobriety},
  pdfkeywords={dcpo, Scott topology, meet-continuity, core-compactness, sobriety, localization}
}
\setlist[enumerate]{label=\textup{(\roman*)},leftmargin=*,topsep=3pt,itemsep=1pt,parsep=0pt}
\setlist[itemize]{leftmargin=*,topsep=3pt,itemsep=1pt,parsep=0pt}
\makeatletter
\renewcommand\section{\@startsection{section}{1}{\z@}%
  {-2.2ex plus -.7ex minus -.3ex}{1ex plus .3ex}%
  {\normalfont\normalsize\bfseries\boldmath}}
\renewcommand\subsection{\@startsection{subsection}{2}{\z@}%
  {-1.7ex plus -.5ex minus -.2ex}{.65ex plus .2ex}%
  {\normalfont\normalsize\itshape}}
\patchcmd{\pprintMaketitle}
  {\normalsize\elsauthors\par\vskip10pt
   \footnotesize\itshape\elsaddress\par\vskip36pt}
  {\ifx\elsauthors\@empty\else
     \normalsize\elsauthors\par\vskip10pt\fi
   \ifx\elsaddress\@empty\else
     \footnotesize\itshape\elsaddress\par\vskip36pt\fi}
  {}{\PackageWarning{compactlayout}{Title spacing patch not applied}}
\makeatother

\AtBeginDocument{%
  \setlength{\abovedisplayskip}{7pt plus 2pt minus 2pt}%
  \setlength{\belowdisplayskip}{7pt plus 2pt minus 2pt}%
  \setlength{\abovedisplayshortskip}{3pt plus 2pt}%
  \setlength{\belowdisplayshortskip}{5pt plus 2pt minus 2pt}%
}
\AtBeginEnvironment{proof}{\setlength{\topsep}{4pt plus 1pt minus 1pt}}

\newtheoremstyle{compactplain}{5pt plus 1pt minus 1pt}{5pt plus 1pt minus 1pt}
  {\itshape}{}{\bfseries}{.}{.5em}{}
\newtheoremstyle{compactdefinition}{5pt plus 1pt minus 1pt}{5pt plus 1pt minus 1pt}
  {}{}{\bfseries}{.}{.5em}{}
\newtheoremstyle{compactremark}{4pt plus 1pt minus 1pt}{4pt plus 1pt minus 1pt}
  {}{}{\itshape}{.}{.5em}{}
\theoremstyle{compactplain}
\newtheorem{theorem}{Theorem}[section]
\newtheorem{lemma}[theorem]{Lemma}
\newtheorem{proposition}[theorem]{Proposition}
\newtheorem{corollary}[theorem]{Corollary}
\theoremstyle{compactdefinition}
\newtheorem{definition}[theorem]{Definition}
\theoremstyle{compactremark}
\newtheorem{remark}[theorem]{Remark}

\DeclareMathOperator{\height}{ht}
\DeclareMathOperator{\loc}{loc}
\newcommand{\NN}{\mathbb N}

\makeatletter
\let\ps@pprintTitle\ps@plain
\makeatother

\begin{document}
\begin{frontmatter}
\title{A non-sober core-compact meet-continuous dcpo}
\author[scu]{Yuxu Chen}
\ead{chenyuxu@scu.edu.cn}
\affiliation[scu]{organization={School of Mathematics, Sichuan University},
  city={Chengdu},
  country={China}}

\begin{abstract}
We construct a meet-continuous dcpo whose Scott space is core-compact
but neither sober nor locally compact.  The example is
a proper sub-dcpo of a locally compact sober complete lattice, and
restriction induces an isomorphism between their lattices of Scott-closed
sets. It follows that a locally compact sober dcpo need not be determined, among all
dcpos, by its Scott-closed-set lattice. 
\end{abstract}
\end{frontmatter}

\section{Introduction}\label{sec:introduction}

A dcpo carries a canonical topology, the Scott topology, which connects
directed approximation with non-Hausdorff topology.  Continuous and
quasicontinuous dcpos have sober Scott spaces, but arbitrary Scott spaces
need not be sober \cite{AbramskyJung1994,CLD2003,Johnstone1981}.  A
natural problem is whether a topological approximation property can
replace continuity.  In particular, Jia asked whether every
core-compact meet-continuous dcpo is sober
\cite[Question~4.5.8]{Jia2018}.

We answer this question negatively.  We construct a meet-continuous dcpo
\(D\) such that \(\Sigma D\) is compact, second countable and
core-compact, but not sober.  He and Zhang recently proved that every
second-countable locally compact dcpo is sober
\cite[Theorem~4.13]{HeZhang2026}.  Hence the same example is not locally
compact.  Thus core-compactness and local compactness differ even for
second-countable Scott spaces of meet-continuous dcpos.

The construction also gives a reconstruction counterexample.  For a
dcpo \(P\), let \(\Gamma(P)\) denote the lattice of Scott-closed subsets.
The dcpo \(P\) is \(C_\sigma\)-unique if
\(\Gamma(P)\cong\Gamma(Q)\) implies \(P\cong Q\) for every dcpo \(Q\)
\cite{ChenHe2026}.  We construct a complete lattice \(L\) and a proper
sub-dcpo \(D\subsetneq L\) for which
\[
  \sigma(L)\cong\sigma(D),\qquad
  \Gamma(L)\cong\Gamma(D).
\]
The lattice \(L\) has a greatest element, whereas \(D\) has none, so
they are not order isomorphic.  Since \(\Sigma L\) is locally compact
and sober, this shows that a locally compact sober dcpo need not be
\(C_\sigma\)-unique.

The proof has three components.  First, an abstract transfer theorem
shows that two elementary approximation properties of a sub-dcpo
\(D\subseteq L\) force restriction to be a frame isomorphism
\(\sigma(L)\cong\sigma(D)\).  Core-compactness, second countability,
meet-continuity and non-sobriety then follow with short 
arguments.  Second, we construct an explicit compact topological lattice
\(L\), inspired by Lawson's examples
\cite{Lawson1970,Lawson1973,GL2022}.  Third, we select \(D\) by a
transfinite recursion.  At each stage we insert either one point or one
increasing sequence converging to a point selected earlier.  A stage
rank then makes closure under directed suprema transparent.

\section{Preliminaries}\label{sec:preliminaries}

Put \(\NN=\{0,1,2,\ldots\}\) and \(\mathfrak c=2^{\aleph_0}\).  All
directed sets are nonempty.  For a subset \(A\) of a poset \(P\), write
\(
 \downarrow_PA=\{x\in P:x\leq a\text{ for some }a\in A\} \) and \(
 \uparrow_PA=\{x\in P:a\leq x\text{ for some }a\in A\}.
\)
Subscripts are omitted when the ambient poset is clear.  
A dcpo is a poset in which every directed subset has a supremum.  A subset
\(D\subseteq P\) of a dcpo is a \emph{sub-dcpo} if every directed
\(E\subseteq D\) has the same supremum in \(D\) and in \(P\).

A subset \(U\) of a dcpo \(P\) is Scott open if it is upper and, whenever
\(E\subseteq P\) is directed with \(\sup E\in U\), one has
\(E\cap U\neq\varnothing\).  The Scott space is denoted by \(\Sigma P\),
its open-set frame by \(\sigma(P)\), and its Scott-closed-set lattice by
\(\Gamma(P)\).  
 A space is
core-compact exactly when its open-set lattice is continuous; local compactness implies core-compactness \cite{GL2013}.

A dcpo \(P\) is \emph{meet-continuous} if, for every directed
\(E\subseteq P\) and every \(p\leq\sup E\),
\(
 p\in\operatorname{cl}_{\Sigma P}
       (\downarrow p\cap\downarrow E).
\)
This definition does not require binary meets
\cite{KouLiuLuo2003,Jia2018}.  A frame isomorphism preserves arbitrary
joins and finite meets, including the empty meet.

\section{A transfer principle}\label{sec:transfer}

The following theorem give some sufficient conditions for a sub-dcpo \(D\subseteq L\)  being meet-continuous, core-compact, second countable and non-sober.

\begin{theorem}\label{thm:transfer}
Let \(L\) be a dcpo with least and greatest elements \(\bot_L\) and
\(\top_L\).  Assume that \(\Sigma L\) is second countable and
core-compact.  Let \(D\subseteq L\) be a sub-dcpo such that
\(\bot_L\in D\) and \(\top_L\notin D\).  Suppose that:
\begin{enumerate}
\item\label{cond:density}
for every \(p\in V\in\sigma(L)\), there is
\(d\in D\cap V\) with \(d\leq p\);
\item\label{cond:common-lower}
if \(U\in\sigma(D)\), \(q\in U\), and \(E\subseteq L\) is directed
with \(q\leq\sup_LE\), then there are \(u\in U\) and \(e\in E\)
such that \(u\leq q\) and \(u\leq e\).
\end{enumerate}
Then restriction is a frame isomorphism
\[
  \rho:\sigma(L)\longrightarrow\sigma(D),\qquad
  \rho(V)=V\cap D,
\]
whose inverse is \(U\mapsto\uparrow_LU\).  Moreover, \(D\) is
meet-continuous and \(\Sigma D\) is compact, second countable,
core-compact and non-sober.  The lattices \(\Gamma(L)\) and
\(\Gamma(D)\) are isomorphic, but \(L\) and \(D\) are not order
isomorphic.
\end{theorem}

\begin{proof}
If \(V\in\sigma(L)\), then \(V\cap D\) is Scott open in \(D\): if a
directed \(E\subseteq D\) has supremum in \(V\cap D\), the sub-dcpo
property gives \(\sup_DE=\sup_LE\), and hence \(E\cap V\neq\varnothing\).

Conversely, let \(U\in\sigma(D)\) and put
\(\widehat U=\uparrow_LU\).  This set is upper.  If a directed
\(E\subseteq L\) satisfies \(\sup_LE\in\widehat U\), choose
\(q\in U\) with \(q\leq\sup_LE\).  Condition
\ref{cond:common-lower} gives \(u\in U\) and \(e\in E\) with
\(u\leq e\), so \(e\in\widehat U\).  Thus \(\widehat U\) is Scott
open in \(L\).

Since \(U\) is upper in \(D\),
\(
 (\uparrow_LU)\cap D=U.
\)
If \(V\in\sigma(L)\), condition \ref{cond:density} gives
\(
 \uparrow_L(V\cap D)=V:
\)
the nontrivial inclusion follows by choosing, for each \(p\in V\), a
point \(d\in D\cap V\) below \(p\).  Hence restriction and upward
extension are inverse.  Restriction preserves arbitrary unions and
finite intersections, so it is a frame isomorphism.  In particular,
the intrinsic Scott topology of \(D\) is the subspace topology inherited
from \(\Sigma L\).

The frame \(\sigma(D)\) is continuous because \(\sigma(L)\) is, so
\(\Sigma D\) is core-compact.  The traces of a countable basis of
\(\Sigma L\) form a countable basis of \(\Sigma D\).  Since
\(\bot_L\in D\), every Scott-open cover of \(D\) contains a member
containing \(\bot_L\), and that upper set is all of \(D\).  Thus
\(\Sigma D\) is compact.

To prove meet-continuity, let \(E\subseteq D\) be directed,
\(q\leq\sup_DE\), and let \(U\) be a Scott-open neighbourhood of
\(q\) in \(D\).  Since \(D\) is a sub-dcpo,
\(\sup_DE=\sup_LE\).  Condition \ref{cond:common-lower} yields
\(u\in U\cap\downarrow q\cap\downarrow E\).  Hence every Scott-open
neighbourhood of \(q\) meets
\(\downarrow q\cap\downarrow E\), which is precisely
meet-continuity.

Any two nonempty Scott-open subsets of \(L\) contain \(\top_L\).
The frame isomorphism therefore implies that any two nonempty
Scott-open subsets of \(D\) intersect.  Thus \(D\), as a closed subset
of itself, is irreducible.  We show that \(D\) has no greatest element.
For \(g\in D\), one has \(g\neq\top_L\).  Since the specialization
order of \(\Sigma L\) is the order of \(L\), there is a Scott-open set
\(V\) with \(\top_L\in V\) and \(g\notin V\).  By condition
\ref{cond:density}, choose \(d\in D\cap V\) with \(d\leq\top_L\).
If \(d\leq g\), upperness of \(V\) would imply \(g\in V\), a
contradiction.  Hence \(g\) is not greatest.

If the irreducible closed set \(D\) had a generic point \(g\), then
\(D=\operatorname{cl}_{\Sigma D}\{g\}=\downarrow_Dg\), making \(g\)
a greatest element.  Therefore \(\Sigma D\) is not sober.  Taking
complements transports the frame isomorphism to an order isomorphism
\(\Gamma(L)\cong\Gamma(D)\).  Finally, \(L\) has a greatest element
and \(D\) does not, so the two dcpos are not order isomorphic.
\end{proof}

\begin{corollary}\label{cor:not-locally-compact}
Under the hypotheses of Theorem~\ref{thm:transfer}, \(\Sigma D\) is not
locally compact.
\end{corollary}

\begin{proof}
The space \(\Sigma D\) is second countable and non-sober by
Theorem~\ref{thm:transfer}.  He and Zhang proved that every
second-countable locally compact dcpo has a sober Scott space
\cite[Theorem~4.13]{HeZhang2026}.  Therefore \(\Sigma D\) cannot be
locally compact.
\end{proof}

\section{The ambient complete lattice
\texorpdfstring{\(L\)}{L}}\label{sec:ambient}

We now construct the ambient complete lattice.  The only features used later are its explicit coordinate order, a compact metrizable auxiliary topology, the description of its Scott topology, and the fact that binary meet preserves directed suprema.

\subsection{The order and the coordinate bounds}\label{subsec:ambient-strategy}

Coordinate indices are integers \(j\geq2\). Define
\begin{equation}\label{eq:blocks}
 d_j=2^{2^j},\qquad N_j=\{1,2,3,\ldots,d_j\}.
\end{equation}
Write
\[
 \mathcal C=\prod_{j\geq2}\mathcal P(N_j).
\]
An element \(A=(A_j)_{j\geq2}\in\mathcal C\) will be called a
\emph{configuration}. Thus a configuration is a sequence of subsets
\(A_j\subseteq N_j\). Equip \(\mathcal C\) with the product order
\[
 A\preceq B\quad\Longleftrightarrow\quad
 A_j\subseteq B_j\text{ for every }j\geq2.
\]
For each \(j\geq2\), define \(c_j:\mathcal P(N_j)\to[0,1]\) by
\begin{equation}\label{eq:cj}
 c_j(E)=\max\left\{0,1-\frac{\log_2(1+|N_j\setminus E|)}{j}\right\}.
\end{equation}
If \(E\subseteq H\subseteq N_j\), then
\(|N_j\setminus H|\leq|N_j\setminus E|\), so \(c_j(E)\leq c_j(H)\).
Also \(c_j(\varnothing)=0\) and \(c_j(N_j)=1\).
Define \(c:\mathcal C\to[0,1]\) by
\[
 c(A)=\inf_{j\geq2}c_j(A_j).
\]
We define
\begin{equation}\label{eq:ambient}
 L=\left\{(t,A)\in[0,1]\times\mathcal C:
                   t\leq c_j(A_j)\text{ for every }j\right\},
\end{equation}
where \(A=(A_j)_{j\geq2}\). For an element \(x=(t,A)\in L\), its
\emph{height} is the value of its first coordinate. Thus
\(\height:L\to[0,1]\) is defined by \(\height(t,A)=t\).
For fixed \(A\in\mathcal C\), one has \((t,A)\in L\) exactly when
\(0\leq t\leq c(A)\); hence \(c(A)\) is the largest permissible height
at \(A\). For coordinate comparisons, we also write
\(x=(t_x,A^x)\), where \(t_x=\height(x)\) and
\(A^x=(A_j^x)_{j\geq2}\in\mathcal C\).
The order is defined by
\begin{equation}\label{eq:order}
 (t,A)\leq(s,B)\quad\Longleftrightarrow\quad
 t\leq s\ \text{and}\ A\preceq B.
\end{equation}
The least and greatest elements are
\[
 \bot_L=(0,(\varnothing)_j),\qquad \top_L=(1,(N_j)_j).
\]
Height is a real coordinate in \([0,1]\), not an integer rank. An inequality
between heights does not by itself imply an order comparison, and points
with the same height need not be incomparable.

For \(0<r\leq1\) and \(j\geq2\), put
\[
 q_r(j)=\lfloor2^{(1-r)j}\rfloor-1.
\]
For every \(E\subseteq N_j\), we have
\begin{equation}\label{eq:complement-bound}
 c_j(E)\geq r\ \Longleftrightarrow\
 |N_j\setminus E|\leq q_r(j).
\end{equation}
Thus positive height bounds the number of omitted elements in every
coordinate. At height \(1\), every coordinate must be \(N_j\), so there is
only the point \(\top_L\). At height \(0\), the condition is automatic and
all sequences of subsets are allowed. The equivalence in
\eqref{eq:complement-bound} is not asserted for \(r=0\).

Let
\(
X=[0,1]\times\mathcal C,
\)
where \([0,1]\) carries its usual topology and each finite set
\(\mathcal P(N_j)\) carries the discrete topology.
Each factor is compact, metrizable, and Hausdorff.
Since \(X\) is a countable product of compact metrizable Hausdorff
spaces, it is compact by Tychonoff's theorem and is also metrizable
and Hausdorff \cite[Sections~21 and~37]{Munkres2000}.

Let \(\tau\) be the subspace topology on \(L\) inherited from \(X\).
If \(x=(t,A)\in L\), then a basic \(\tau\)-neighbourhood of \(x\)
is of the form
\[
U(x;I,J)
=
\left\{
(s,B)\in L:
s\in I,\ B_j=A_j\text{ for every }j\in J
\right\},
\]
where \(I\) is a relatively open interval in \([0,1]\) containing \(t\),
and \(J\subseteq\{2,3,\ldots\}\) is finite.
Indeed, the product topology allows restrictions on only finitely
many coordinates, and each discrete factor \(\mathcal P(N_j)\) has
the singleton \(\{A_j\}\) as an open neighbourhood of \(A_j\).

The functions \(c_j\) are chosen for two reasons.
First, deleting one element decreases \(c_j\) by at most \(1/j\), and
for \(E,H\subseteq N_j\), the value \(c_j(E\cap H)\) differs from
\(\min\{c_j(E),c_j(H)\}\) by at most \(1/j\); see
Lemma~\ref{lem:loss}.
This allows changes in all sufficiently large coordinates after an arbitrarily small decrease of
height. It also makes binary meet continuous for \(\tau\).
Second, the exponential bound in \eqref{eq:complement-bound}, together with the larger sets
\(N_j\), gives the cardinal estimates used in the selection of \(D\).

\begin{lemma}\label{lem:loss}
For \(E,H\subseteq N_j\) and \(z\in N_j\),
\begin{align}
  0&\leq c_j(E)-c_j(E\setminus\{z\})\leq\frac1j,
       \label{eq:deletion-loss}\\
  0&\leq\min\{c_j(E),c_j(H)\}-c_j(E\cap H)\leq\frac1j.
       \label{eq:intersection-loss}
\end{align}
\end{lemma}

\begin{proof}
The map \(u\mapsto\max\{0,u\}\) is increasing and \(1\)-Lipschitz.
If \(z\in E\) and \(k=|N_j\setminus E|\), then
\(|N_j\setminus(E\setminus\{z\})|=k+1\).
Thus the decrease in \(c_j\) is at most
\[
\left(1-\frac1j\log_2(1+k)\right)
-\left(1-\frac1j\log_2(2+k)\right)
=\frac1j\log_2\frac{k+2}{k+1}
\leq\frac1j.
\]
If \(z\notin E\), there is no change.

For intersection, put
\(k=\max\{|N_j\setminus E|,|N_j\setminus H|\}\)
and \(m=|N_j\setminus(E\cap H)|\).
Then \(k\leq m\leq2k\).
The smaller of \(c_j(E)\) and \(c_j(H)\)
corresponds to \(k\), so the decrease is at most
\[
\left(1-\frac1j\log_2(1+k)\right)
-\left(1-\frac1j\log_2(1+m)\right)
=\frac1j\log_2\frac{1+m}{1+k}
\leq\frac1j\log_2\frac{1+2k}{1+k}
\leq\frac1j.
\]
Monotonicity gives the lower bounds. The truncation at zero preserves
all these estimates.
\end{proof}

Since the conditions defining \(L\) are imposed coordinatewise,
deleting one element from each of infinitely many coordinates
does not require adding up the decreases in the coordinate height bounds.
Indeed, let \((t,A)\in L\) and \(M\geq2\). Suppose that
\(B_j=A_j\) for \(j<M\), while \(B_j\) is obtained from \(A_j\)
by deleting at most one element for each \(j\geq M\).
Then
\[
c_j(B_j)\geq
\begin{cases}
t, & j<M,\\[2mm]
t-\dfrac1j\geq t-\dfrac1M, & j\geq M.
\end{cases}
\]
Consequently,
\[
\left(\max\{0,t-1/M\},B\right)\in L.
\]
Thus, by choosing \(M\) sufficiently large, we can delete
elements from infinitely many coordinates while decreasing
the height by an arbitrarily small amount.

\subsection{The lattice and its compact Hausdorff topology}

All operations below are taken in the coordinate order of \(L\), and
continuity refers to the auxiliary topology \(\tau\) until the Scott topology
is identified.

\Needspace{14\baselineskip}
\begin{lemma}\label{lem:lattice}
The poset \(L\) is a complete lattice. For a nonempty family
\(\{(t_\alpha,A^\alpha):\alpha\in I\}\subseteq L\),
\begin{equation}\label{eq:join}
  \bigvee_{\alpha\in I}(t_\alpha,A^\alpha)
  =\left(\sup_{\alpha\in I}t_\alpha,
                 \left(\bigcup_{\alpha\in I}A^\alpha_j\right)_j\right).
\end{equation}
Writing \(B^\wedge_j=\bigcap_{\alpha\in I}A^\alpha_j\), its meet is
\begin{equation}\label{eq:meet}
  \bigwedge_{\alpha\in I}(t_\alpha,A^\alpha)
   =\left(\min\left\{\inf_{\alpha\in I}t_\alpha,
                 \inf_{j\geq2}c_j(B^\wedge_j)\right\},B^\wedge\right),
\end{equation}
where \(B^\wedge=(B^\wedge_j)_{j\geq2}\).
The least and greatest elements are the points \(\bot_L\) and \(\top_L\)
introduced in Subsection~\ref{subsec:ambient-strategy}.
Moreover, \((L,\tau)\) is compact metrizable Hausdorff, the order is closed,
and both binary lattice operations are jointly \(\tau\)-continuous.
\end{lemma}

\begin{proof}
Let \(\{(t_\alpha,A^\alpha)\}_{\alpha\in I}\) be a nonempty family
in \(L\), and put
\[
 B^\vee_j=\bigcup_{\alpha\in I}A_j^\alpha,
 \qquad t_\vee=\sup_{\alpha\in I}t_\alpha.
\]
Write \(B^\vee=(B^\vee_j)_{j\geq2}\).
For every \(j\geq2\) and \(\alpha\in I\), monotonicity gives
\(c_j(B^\vee_j)\geq c_j(A_j^\alpha)\geq t_\alpha\).
Hence \(c_j(B^\vee_j)\geq t_\vee\), so
\((t_\vee,B^\vee)\in L\). This point is an upper bound of the
family. If \((s,C)\in L\) is any other upper bound, then
\(t_\alpha\leq s\) and \(A_j^\alpha\subseteq C_j\) for all
\(\alpha,j\). Thus \(t_\vee\leq s\) and
\(B^\vee_j\subseteq C_j\) for every \(j\), proving
\eqref{eq:join}.

For the meet, put
\[
 B^\wedge_j=\bigcap_{\alpha\in I}A_j^\alpha,
 \qquad
 t_\wedge=\min\left\{\inf_{\alpha\in I}t_\alpha,
                       \inf_{j\geq2}c_j(B^\wedge_j)\right\}.
\]
Then \((t_\wedge,B^\wedge)\in L\), and it lies below every member
of the family. If \((s,C)\in L\) is another lower bound, then
\(s\leq t_\alpha\) and \(C_j\subseteq A_j^\alpha\) for all
\(\alpha,j\). Consequently \(C_j\subseteq B^\wedge_j\) and
\(s\leq c_j(C_j)\leq c_j(B^\wedge_j)\) for every \(j\), while
\(s\leq\inf_\alpha t_\alpha\). Hence \(s\leq t_\wedge\) and
\((s,C)\leq(t_\wedge,B^\wedge)\), proving \eqref{eq:meet}.
The points \(\bot_L\) and \(\top_L\) are respectively the empty
join and the empty meet, so \(L\) is a complete lattice.

We next show that \(L\) is closed in \(X\).
Let \(\pi_0:X\to[0,1]\) and \(\pi_j:X\to\mathcal P(N_j)\) be the
coordinate projections, \(\pi_0(t,A)=t\) and \(\pi_j(t,A)=A_j\).
They are continuous in the product topology. The map
\(c_j:\mathcal P(N_j)\to\mathbb R\) is continuous because its
domain is discrete. Hence \(f_j=\pi_0-c_j\circ\pi_j\) is a
continuous real-valued function on \(X\), and
\[
 \{(t,A)\in X:t\leq c_j(A_j)\}=f_j^{-1}((-\infty,0])
\]
is closed. Taking the intersection over all \(j\geq2\) proves that
\(L\) is closed in \(X\). Since \(X\) is compact metrizable Hausdorff,
so is its closed subspace \((L,\tau)\).

The order relation on \(L\) is given by
\[
(s,A)\leq(t,B)
\iff
s\leq t
\quad\text{and}\quad
A_j\subseteq B_j\text{ for every }j\geq2.
\]
The relation
\[
R_0=\{(s,t)\in[0,1]^2:s\leq t\}
\]
is closed, since it is the inverse image of the closed set
\((-\infty,0]\) under the continuous map
\[
(s,t)\longmapsto s-t.
\]
For each \(j\geq2\), the relation
\[
R_j=
\{(A_j,B_j)\in\mathcal P(N_j)^2:A_j\subseteq B_j\}
\]
is closed because \(\mathcal P(N_j)^2\) is discrete.

Let \(\Pi_0:X\times X\to[0,1]^2\) and
\(\Pi_j:X\times X\to\mathcal P(N_j)^2\) be the continuous maps
\[
 \Pi_0((s,A),(t,B))=(s,t),\qquad
 \Pi_j((s,A),(t,B))=(A_j,B_j).
\]
Then the coordinate order on \(X\) has graph
\[
 G_X=\Pi_0^{-1}(R_0)\cap\bigcap_{j\geq2}\Pi_j^{-1}(R_j).
\]
Each set in this intersection is closed, since the coordinate projections are continuous. Hence \(G_X\) is closed in \(X\times X\).
Therefore
\[
G_L=G_X\cap(L\times L)
\]
is closed in \(L\times L\).

By \eqref{eq:join}, binary join is given by
\[
(s,A)\vee(t,B)
=
\left(\max\{s,t\},(A_j\cup B_j)_{j\geq2}\right).
\]
The maximum function is continuous on \([0,1]^2\), and the map
\((E,H)\mapsto E\cup H\) on \(\mathcal P(N_j)^2\) is continuous,
since that domain is discrete. Composing with the coordinate
projections shows that every output coordinate of join is continuous.
Thus join is continuous as a map \(L\times L\to X\); its image
lies in \(L\), so the subspace topology makes it jointly
\(\tau\)-continuous as a map into \(L\).

It remains to prove that binary meet is continuous. Define
\(m,m_n:L\times L\to[0,1]\), for \(n\geq2\), by
\[
 \begin{aligned}
 m(x,y)&=\height(x\wedge y),\\
 m_n(x,y)&=\min\{t_x,t_y,c_2(A_2^x\cap A_2^y),\ldots,
                                  c_n(A_n^x\cap A_n^y)\}.
 \end{aligned}
\]
Each \(m_n\) is continuous, since it is the minimum of finitely many
continuous functions. For \(j>n\), Lemma~\ref{lem:loss} gives
\[
 \begin{aligned}
 c_j(A_j^x\cap A_j^y)
 &\geq\min\{c_j(A_j^x),c_j(A_j^y)\}-\frac1j\\
 &\geq\min\{t_x,t_y\}-\frac1j
 \geq m_n(x,y)-\frac1{n+1}.
 \end{aligned}
\]
By \eqref{eq:meet},
\[
 m(x,y)=\min\left\{m_n(x,y),
                    \inf_{j>n}c_j(A_j^x\cap A_j^y)\right\}.
\]
Consequently
\[
 0\leq m_n(x,y)-m(x,y)\leq\frac1{n+1}
 \qquad(x,y\in L).
\]
Thus \(m_n\) converges uniformly to \(m\), and \(m\) is continuous
by the uniform limit theorem \cite[Theorem~21.6]{Munkres2000}.
For each \(j\), the \(j\)-th configuration coordinate of the meet is
\(A_j^x\cap A_j^y\). It is continuous because intersection is a map
between finite discrete spaces. Hence every coordinate of binary meet
is continuous, and its values lie in \(L\). The product and subspace
topologies therefore make binary meet jointly \(\tau\)-continuous.
\end{proof}

\subsection{The intrinsic Scott topology}

The following identification is an application of Lawson's theorem on
compact Hausdorff topological semilattices
\cite[Theorem~13, pp.~599--600]{Lawson1973}.
The term \emph{ideal} in that theorem means a lower subset; directedness
of that subset is not required.

\begin{lemma}\label{lem:scott-ambient}
The Scott topology of \(L\) consists exactly of the \(\tau\)-open upper sets.
For each \(x\in L\), the map \(y\mapsto x\wedge y\) preserves directed
suprema.
\end{lemma}

\begin{proof}
By Lemma~\ref{lem:lattice}, \((L,\tau)\) is a compact Hausdorff
topological lattice and hence a topological meet-semilattice. Lawson's theorem states that a lower subset of such a
semilattice is closed if and only if it contains the suprema of its
nondecreasing nets whenever those suprema exist. A directed subset,
indexed by its own order, is a nondecreasing net; conversely, the range
of a nondecreasing net is directed. Hence a lower subset of \(L\) is
\(\tau\)-closed if and only if it is closed under directed suprema,
that is, Scott closed. Taking complements proves the first assertion.

For the second assertion, fix \(x\in L\) and a directed
\(E\subseteq L\). By Lawson's monotone-net lemma
\cite[Lemma~9, p.~598]{Lawson1973}, the net indexed by \((E,\leq)\)
converges in \(\tau\) to \(\sup_LE\). Continuity of
\(y\mapsto x\wedge y\) therefore gives convergence of its image net to
\(x\wedge\sup_LE\). The image net is also nondecreasing, so the same
lemma makes it converge to \(\sup_{e\in E}(x\wedge e)\).
Uniqueness of limits in the Hausdorff topology \(\tau\) yields
\begin{equation}\label{eq:ambient-meet-continuity}
 x\wedge\sup_LE=\sup_{e\in E}(x\wedge e).\qedhere
\end{equation}
\end{proof}

\begin{proposition}\label{prop:ambient-properties}
The Scott space of \(L\) is compact, second countable, locally compact,
core-compact and sober. Its Scott topology has the countable basis
\begin{equation}\label{eq:basis}
 \begin{aligned}
 &\{(t,A)\in L: E_j\subseteq A_j\text{ for all }j\in J \},\\
 &\{(t,A)\in L:t>r,\ E_j\subseteq A_j\text{ for all }j\in J \},
 \end{aligned}
\end{equation}
where \(J\subseteq\{2,3,\ldots\}\) is finite, \(E_j\subseteq N_j\), and
\(r\in\mathbb Q\cap[0,1)\). These conditions require elements to be present,
not equality of coordinate subsets.
\end{proposition}

\begin{proof}
For \(j\geq2\), \(z\in N_j\), and \(r\in\mathbb Q\cap[0,1)\), put
\[
 O_{j,z}=\{(t,A)\in L:z\in A_j\},\qquad
 H_r=\{(t,A)\in L:t>r\}.
\]
Both sets are upper. If a directed set \(E\subseteq L\) has supremum
in \(O_{j,z}\), formula~\eqref{eq:join} shows that \(z\) belongs to the
\(j\)-th coordinate of some member of \(E\). Thus \(E\cap O_{j,z}\ne\varnothing\).
If \(\sup_LE\in H_r\), then
\(\sup_{e\in E}\height(e)>r\), so some \(e\in E\) has
\(\height(e)>r\). Thus \(H_r\) is Scott open as well.
The sets in \eqref{eq:basis} are finite intersections of these opens.
There are countably many possible finite sets of coordinate requirements
and countably many rational thresholds, so this family is countable.

We prove that it is a basis. Let \(x\in U\in\sigma(L)\).
Lemma~\ref{lem:scott-ambient} makes \(U\) a \(\tau\)-open upper set,
so \(K_0=L\setminus U\) is \(\tau\)-compact. For each \(y\in K_0\),
one has \(x\nleq y\), because \(x\leq y\) would imply \(y\in U\).
If \(z\in A_j^x\setminus A_j^y\) for some \(j,z\), choose
\(O_y=O_{j,z}\) and \(W_y=L\setminus O_{j,z}\).
Then \(x\in O_y\), \(y\in W_y\), and \(W_y\) is \(\tau\)-open.
If no such coordinate exists, \eqref{eq:order} implies
\(\height(y)<\height(x)\). Choose a rational \(r\) strictly between
these heights and put
\[
 O_y=H_r,\qquad W_y=\{z\in L:\height(z)<r\}.
\]
Again \(x\in O_y\), \(y\in W_y\), and \(O_y\cap W_y=\varnothing\).
Compactness gives \(y_1,\ldots,y_m\) such that
\(K_0\subseteq\bigcup_{i=1}^mW_{y_i}\). The finite intersection
\(\bigcap_{i=1}^mO_{y_i}\) contains \(x\) and is disjoint from \(K_0\).
Combining the finitely many required entries and taking the largest
height threshold, if there is one, expresses this intersection in the
form \eqref{eq:basis}. If \(K_0=\varnothing\), use the basis member \(L\).
Thus every \(x\in U\) belongs to a basis member contained in \(U\).

For local compactness, choose such a basis member \(W\) with
\(x\in W\subseteq U\). If \(W\) has no height threshold, put \(V=K=W\).
The finite coordinate conditions make \(K\) \(\tau\)-closed and upper.
If \(W\) has height threshold \(r\), choose a rational \(r'\) with
\(r<r'<\height(x)\) and put
\[
 \begin{aligned}
 V&=\{(t,A)\in L:t>r',\ E_j\subseteq A_j\ (j\in J)\},\\
 K&=\{(t,A)\in L:t\geq r',\ E_j\subseteq A_j\ (j\in J)\}.
 \end{aligned}
\]
In both cases \(x\in V\subseteq K\subseteq U\), \(V\) is Scott open,
and \(K\) is \(\tau\)-closed and upper. Since \((L,\tau)\) is compact,
\(K\) is \(\tau\)-compact. Every Scott-open set is \(\tau\)-open by
Lemma~\ref{lem:scott-ambient}, so \(K\) is also a compact subset of \(\Sigma L\).
Its upperness makes it saturated in \(\Sigma L\). This proves local
compactness, hence also core-compactness, of \(\Sigma L\).

It is compact because an open cover has a member containing \(\bot_L\),
and that upper set equals \(L\).

Finally, since \(L\) is a complete lattice and its Scott space is core-compact, it is sober \cite[Corollary~3.3]{JiaJungLi2016}. 
\end{proof}

\section{The tail-selection lemma}\label{sec:selection}

Write
\[
 L_+=\{x\in L:\height(x)>0\}.
\]
The recursion needs only the following two consequences of the finite
slalom argument.  The first supplies a small family of reference points
which is locally cofinal from below.  The second allows one tail function
to avoid any previously selected family of smaller cardinality.

\begin{proposition}[Tail selection]\label{prop:tail-selection}
There are an uncountable cardinal \(\kappa\leq\mathfrak c\) and a set
\(\mathcal R\subseteq L_+\), with \(|\mathcal R|\leq\kappa\), such
that:
\begin{enumerate}
\item\label{sel:reference}
if \(p\in V\in\sigma(L)\) and \(\height(p)>0\), then some
\(b\in\mathcal R\cap V\) satisfies \(b\leq p\);
\item\label{sel:avoidance}
if \(H\subseteq L_+\) and \(|H|<\kappa\), then there is
\(f\in\prod_{j\geq2}N_j\) such that, for every
\(x=(t_x,A^x)\in H\),
\[
 \{j\geq2:f(j)\in A_j^x\}
\]
is infinite.
\end{enumerate}
\end{proposition}

The proof is given in Appendix~\ref{app:selection}.  No other property
of \(\kappa\) or \(\mathcal R\) will be used.

\begin{remark}\label{rem:CH}
Under the continuum hypothesis, the cardinal in
Proposition~\ref{prop:tail-selection} is \(\aleph_1\).  At every stage
of the recursion below, only countably many points have then been
selected, and condition \ref{sel:avoidance} follows from the ordinary
countable diagonal argument.  The slalom proof in the appendix is the
ZFC replacement for this simpler picture.
\end{remark}

\section{The recursively selected sub-dcpo}\label{sec:recursive}

Fix \(\kappa\) and \(\mathcal R\) as in
Proposition~\ref{prop:tail-selection}, and regard \(\kappa\) as its
initial ordinal.  Put
\[
 Z=\{0\}\times\prod_{j\geq2}\mathcal P(N_j),
\]
the zero-height layer of \(L\).  The following insertion lemma is the
only place where condition \ref{sel:avoidance} is used in the main
proof.

\begin{lemma}\label{lem:insertion}
Let \(H\subseteq L_+\) have cardinality less than \(\kappa\).
\begin{enumerate}
\item\label{ins:point}
If \(b\in V\in\sigma(L)\) has positive height, there is
\(s\in V\cap\downarrow_Lb\) with \(0<\height(s)<1\) such that
\(x\nleq s\) for every \(x\in H\).
\item\label{ins:sequence}
Let \(p=(t,A)\in H\), where \(0<t<1\), and let
\(b=(v,C)\in L\) satisfy \(v\geq t\).  There is a strictly increasing
sequence \((r_n)_{n\geq1}\) of positive-height points such that
\[
 r_n\uparrow p
 \quad\text{and}\quad
 x\nleq r_n\quad(x\in H,n\geq1).
\]
Moreover, if \(E\subseteq L\) is directed,
\(p\leq\sup_LE\), and \(b\leq e_0\in E\), then every \(r_n\) lies
below some member of \(E\).
\end{enumerate}
\end{lemma}

\begin{proof}
Choose \(f\in\prod_{j\geq2}N_j\) as in
Proposition~\ref{prop:tail-selection}\ref{sel:avoidance}.  Thus, for
every \(x=(t_x,A^x)\in H\),
\begin{equation}\label{eq:tail-hit}
 f(j)\in A_j^x\quad\text{for infinitely many }j.
\end{equation}
Any configuration omitting \(f(j)\) in all sufficiently large
coordinates therefore fails to contain the configuration of every
member of \(H\).

For \ref{ins:point}, write \(b=(t,A)\).  By
Lemma~\ref{lem:scott-ambient}, \(V\) is \(\tau\)-open.  Choose a basic
\(\tau\)-neighbourhood of \(b\), contained in \(V\), which fixes only
a finite set \(J\) of configuration coordinates and restricts the
height to an interval \(I\ni t\).  Choose \(u\in I\cap(0,t)\), and
then choose \(M\) larger than every element of \(J\), with
\(1/M<t-u\).  Set
\[
 B_j=\begin{cases}
 A_j,&j<M,\\
 A_j\setminus\{f(j)\},&j\geq M.
 \end{cases}
\]
For \(j<M\), one has \(c_j(B_j)\geq t>u\).  For \(j\geq M\),
Lemma~\ref{lem:loss} gives
\(c_j(B_j)\geq t-1/j\geq t-1/M>u\).  Hence \(s=(u,B)\in L\),
with \(s\leq b\), and the chosen prefix and height interval give
\(s\in V\).  Equation~\eqref{eq:tail-hit} implies
\(x\nleq s\) for every \(x\in H\).

For \ref{ins:sequence}, put \(u_n=t(1-2^{-n})\).  Choose strictly
increasing integers \(M_n\) tending to infinity and satisfying
\(2/M_n<t-u_n\).  Define
\begin{equation}\label{eq:recursive-row}
 r_n=(u_n,R^n),\qquad
 R_j^n=\begin{cases}
 A_j,&j<M_n,\\
 (A_j\cap C_j)\setminus\{f(j)\},&j\geq M_n.
 \end{cases}
\end{equation}
For \(j\geq M_n\), Lemma~\ref{lem:loss} gives
\[
 c_j(R_j^n)
 \geq\min\{c_j(A_j),c_j(C_j)\}-\frac2j
 \geq t-\frac2{M_n}>u_n.
\]
The earlier coordinates have capacity at least \(t\), so \(r_n\in L\)
and \(r_n\leq p\).  When the cutoff moves from \(M_n\) to
\(M_{n+1}\), coordinates can only increase, and the heights increase
strictly.  Hence \(r_n<r_{n+1}\).  Each fixed coordinate eventually
equals \(A_j\), while \(u_n\uparrow t\); formula~\eqref{eq:join}
therefore gives \(\sup_nr_n=p\).  Equation~\eqref{eq:tail-hit} again
implies \(x\nleq r_n\) for every \(x\in H\).

Assume now that \(p\leq\sup_LE\) and \(b\leq e_0\in E\).  For a fixed
\(n\), only finitely many elements occur in the coordinates
\(A_j\), \(2\leq j<M_n\).  Each such element occurs in the
corresponding coordinate of some member of \(E\), because
\(p\leq\sup_LE\).  Directedness gives a single \(e\in E\) above
those finitely many members and above \(e_0\).  The finite prefix of
\(r_n\) is contained in that of \(e\); its tail is contained in
\(C\), hence in the configuration of \(e_0\leq e\); and
\(\height(e)\geq v\geq t>u_n\).  Therefore \(r_n\leq e\).
\end{proof}

We now arrange all future insertions by a single bookkeeping sequence.
Let \(\mathcal B\) be the countable Scott basis in
Proposition~\ref{prop:ambient-properties}, and set
\(\mathcal R^*=\mathcal R\cup\{\top_L\}\).  Consider the requirements
\[
\begin{aligned}
 \mathcal I_0&=\{(0,b,V):V\in\mathcal B,\ b\in\mathcal R\cap V\},\\
 \mathcal I_1&=\{(1,\delta,b):\delta\in\kappa\times\NN,
                                      \ b\in\mathcal R^*\}.
\end{aligned}
\]
Their union has cardinality at most \(\kappa\).  Fix a sequence
\((R_\alpha)_{\alpha<\kappa}\) of requirements such that every
requirement occurs cofinally often.  Such a sequence exists because
\(\kappa\times(\mathcal I_0\cup\mathcal I_1)\) has cardinality
\(\kappa\); project a bijection from \(\kappa\) onto this product to
the second coordinate.

At stage \(\alpha<\kappa\), a set \(T_\alpha\) will be either empty, a
singleton, or the range of one strictly increasing sequence.  Write
\[
 T_\alpha=\{p_{\alpha,n}:n\in J_\alpha\},\qquad
 J_\alpha\in\{\varnothing,\{0\},\NN\}.
\]
If a sequence \((r_n)_{n\geq1}\) is inserted, set
\(p_{\alpha,n}=r_{n+1}\).  For \(\delta=(\beta,n)\) with
\(n\in J_\beta\), write \(p_\delta=p_{\beta,n}\).

Suppose that the stages below \(\alpha\) have been constructed, and put
\[
 H_\alpha=\bigcup_{\beta<\alpha}T_\beta.
\]
Since each \(T_\beta\) is countable,
\begin{equation}\label{eq:stage-size}
 |H_\alpha|\leq\max\{|\alpha|,\aleph_0\}<\kappa.
\end{equation}
If \(R_\alpha=(0,b,V)\in\mathcal I_0\), apply
Lemma~\ref{lem:insertion}\ref{ins:point} to \(H_\alpha\), and let
\(T_\alpha\) be the resulting singleton.  If
\(R_\alpha=(1,\delta,b)\in\mathcal I_1\), and if \(\delta\) indexes a
point \(p_\delta\) selected at an earlier stage with
\(\height(p_\delta)\leq\height(b)\), apply
Lemma~\ref{lem:insertion}\ref{ins:sequence} with
\(p=p_\delta\), and let \(T_\alpha\) be the range of the resulting
sequence.  In all other cases put \(T_\alpha=\varnothing\).
Equation~\eqref{eq:stage-size} justifies every application of the
insertion lemma, including at limit stages.

Define
\begin{equation}\label{eq:D-definition}
 D=Z\cup\bigcup_{\alpha<\kappa}T_\alpha.
\end{equation}
All selected points have height strictly between \(0\) and \(1\), so
\(\top_L\notin D\).  Moreover, the insertion lemma gives the fundamental
stage property
\begin{equation}\label{eq:stage-order}
 \alpha_x<\alpha_y\quad\Longrightarrow\quad x\nleq y
 \qquad(x,y\in D\setminus Z),
\end{equation}
where \(\alpha_x\) denotes the unique stage at which \(x\) was
selected.  Uniqueness follows because a newly inserted point cannot
equal an old one.  Accordingly, every \(x\in D\setminus Z\) has a
unique index \(\delta_x=(\alpha_x,n_x)\) with
\(x=p_{\delta_x}\).

\begin{proposition}\label{prop:D-subdcpo}
The set \(D\) is a sub-dcpo of \(L\).
\end{proposition}

\begin{proof}
Let \(E\subseteq D\) be directed.  If every member of \(E\) has height
zero, formula~\eqref{eq:join} shows that \(\sup_LE\in Z\subseteq D\).
Otherwise choose a positive-height point \(x\in E\) whose selection
stage \(\alpha_x\) is least among the positive-height members of \(E\).
Then
\[
 E_x=E\cap\uparrow_Dx
\]
is directed and cofinal in \(E\).  Every \(y\in E_x\) has positive
height.  Minimality gives \(\alpha_y\geq\alpha_x\), while
\eqref{eq:stage-order} and \(x\leq y\) exclude
\(\alpha_x<\alpha_y\).  Hence \(E_x\subseteq T_{\alpha_x}\).

If \(T_{\alpha_x}\) is a singleton, its point is the supremum of
\(E_x\).  Otherwise \(T_{\alpha_x}\) is the range of a strictly
increasing sequence \((r_n)\) whose supremum is a point selected at an
earlier stage.  The set of indices of terms belonging to \(E_x\) either
has a largest element or is unbounded in \(\NN\).  In the first case
\(E_x\) has a greatest element; in the second its supremum is
\(\sup_nr_n\), which belongs to \(D\).  Thus \(\sup_LE_x\in D\).
Cofinality gives \(\sup_LE=\sup_LE_x\), proving that \(D\) is a
sub-dcpo.
\end{proof}

\begin{proposition}\label{prop:two-conditions}
The sub-dcpo \(D\) satisfies conditions
\ref{cond:density} and \ref{cond:common-lower} of
Theorem~\ref{thm:transfer}.
\end{proposition}

\begin{proof}
For condition \ref{cond:density}, let \(p\in V\in\sigma(L)\).  If
\(p\in Z\), use \(d=p\).  Otherwise choose a basis member
\(W\in\mathcal B\) with \(p\in W\subseteq V\).  By
Proposition~\ref{prop:tail-selection}\ref{sel:reference}, choose
\(b\in\mathcal R\cap W\) with \(b\leq p\).  The requirement
\((0,b,W)\) occurs cofinally often, and each occurrence inserts a point
\(d\in W\cap\downarrow b\).  Hence \(d\in D\cap V\) and
\(d\leq p\).

For condition \ref{cond:common-lower}, let \(q\in U\in\sigma(D)\), and
let \(E\subseteq L\) be directed with \(q\leq\sup_LE\).  Suppose first
that \(q\in Z\).  The directed set
\[
 F=\{q\wedge e:e\in E\}
\]
lies in \(Z\subseteq D\).  By
\eqref{eq:ambient-meet-continuity},
\[
 \sup_LF=q\wedge\sup_LE=q.
\]
Since \(D\) is a sub-dcpo, \(\sup_DF=q\).  Scott openness of \(U\)
gives some \(q\wedge e\in U\), which is below both \(q\) and \(e\).

Now suppose that \(q\) has positive height.  Once \(q=p_\delta\) has
been selected, the requirement \((1,\delta,\top_L)\) is eligible and
occurs again.  It inserts a sequence \(p_n\uparrow q\) in \(D\).
Choose a term \(p\in U\); then \(\height(p)<\height(q)\).  Since
\(q\leq\sup_LE\), choose \(e_0\in E\) with
\(\height(e_0)>\height(p)\).  Choose a rational number \(r\) with
\(\height(p)<r<\height(e_0)\), and apply
Proposition~\ref{prop:tail-selection}\ref{sel:reference} to
\(e_0\) and the Scott-open set
\(H_r=\{x\in L:\height(x)>r\}\).  We obtain
\(b\in\mathcal R\) such that
\[
 b\leq e_0,
 \qquad
 \height(b)>r>\height(p).
\]
The requirement \((1,\delta_p,b)\) occurs again and inserts a
sequence \(r_k\uparrow p\).  Some \(r_k\) lies in \(U\).  The final
part of Lemma~\ref{lem:insertion}\ref{ins:sequence}, applied to
\(p\leq q\leq\sup_LE\) and \(b\leq e_0\in E\), gives
\(r_k\leq e\) for some \(e\in E\).  Since also \(r_k\leq p\leq q\),
this is the required common lower point.
\end{proof}

\begin{theorem}\label{thm:main}
There exist a complete lattice \(L\) and a proper sub-dcpo
\(D\subsetneq L\) with the following properties.
\begin{enumerate}
\item \(\Sigma L\) is compact, second countable, locally compact,
core-compact and sober.
\item The dcpo \(D\) is meet-continuous, and \(\Sigma D\) is compact,
second countable and core-compact, but neither sober nor locally compact.
\item Restriction induces isomorphisms
\[
 \sigma(L)\cong\sigma(D),\qquad
 \Gamma(L)\cong\Gamma(D),
\]
but \(L\not\cong D\).
\end{enumerate}
Consequently, not every core-compact meet-continuous dcpo is sober, and
a locally compact sober dcpo need not be \(C_\sigma\)-unique.
\end{theorem}

\begin{proof}
Take the complete lattice \(L\) of Section~\ref{sec:ambient} and the
sub-dcpo \(D\) of Section~\ref{sec:recursive}.  The properties of
\(L\) are given by Proposition~\ref{prop:ambient-properties}.
Proposition~\ref{prop:D-subdcpo} shows that \(D\) is a sub-dcpo;
clearly \(\bot_L\in Z\subseteq D\), while all selected positive-height
points have height below \(1\), so \(\top_L\notin D\).
Proposition~\ref{prop:two-conditions} verifies the two hypotheses of
Theorem~\ref{thm:transfer}.  That theorem gives meet-continuity,
compactness, second countability, core-compactness, non-sobriety and the
two lattice isomorphisms.  Corollary~\ref{cor:not-locally-compact}
gives failure of local compactness.  Finally, \(L\) has a greatest
element and \(D\) does not, so they are not order isomorphic.
\end{proof}

\appendix
\renewcommand{\thetheorem}{\Alph{section}.\arabic{theorem}}
\section{Proof of the tail-selection lemma}\label{app:selection}

This appendix proves Proposition~\ref{prop:tail-selection}.  It is the
only part of the argument using finite-set localization.  Readers
interested only in the domain-theoretic construction may take that
proposition as a black box.

We use standard cardinal arithmetic for infinite cardinals:
\(\mu\cdot\aleph_0=\mu\) and \(\mu^n=\mu\) for every integer
\(n\geq1\) \cite[Chapter~3]{Jech2003}.

\begin{definition}\label{def:localization}
Let \(Y=(Y_j)_{j\geq2}\) be a sequence of nonempty finite sets, and
let \(w=(w_j)_{j\geq2}\) be a sequence of positive integers.
A width-\(w\) slalom on \(Y\) is a sequence \(\psi\) such that
\(\psi(j)\subseteq Y_j\) and \(|\psi(j)|\leq w_j\) for every \(j\).
It localizes, or captures, \(y\in\prod_{j\geq2}Y_j\) if
\(y(j)\in\psi(j)\) for all but finitely many \(j\). A family is capturing if every point of this
product is captured by some member of the family. Write \(\loc(Y,w)\)
for the least cardinality of a width-\(w\) capturing family on \(Y\).
\end{definition}

The coordinate sets \(\psi(j)\) may be empty. In the notation of
\cite[Definition~3.1]{CardonaMejia2023}, \(\loc(Y,w)\) is
\(\mathfrak d^{\mathrm{Lc}}_{Y,w}\), with the coordinates indexed
by \(\{2,3,\ldots\}\). Set
\begin{equation}\label{eq:lambda}
 \lambda=\loc\bigl((N_j)_{j\geq2},(2^j)_{j\geq2}\bigr).
\end{equation}

The lower bound below is the standard diagonal argument for slalom
covers; compare \cite[Fact~1.4]{GoldsternShelah1993}. We include it
for the eventual-membership convention used here.

\begin{lemma}\label{lem:lambda-bounds}
One has \(\aleph_1\leq\lambda\leq\mathfrak c\).
\end{lemma}

\begin{proof}
Let \((\psi_m)_{m\geq1}\) be a countable family of slaloms on
\((N_j)_{j\geq2}\), with \(|\psi_m(j)|\leq2^j\). A nonempty finite
family can also be listed in this way by repetition. Partition
\(\{2,3,\ldots\}\) into pairwise disjoint infinite sets \(J_m\).
For \(j\in J_m\), choose \(x(j)\in N_j\setminus\psi_m(j)\).
This is possible because \(2^j<|N_j|=2^{2^j}\). These choices define
\(x\in\prod_{j\geq2}N_j\), and \(\psi_m\) fails to capture \(x\)
at all indices in \(J_m\). Thus no countable family is capturing;
the empty family is not capturing either. This proves
\(\lambda\geq\aleph_1\).

For each \(x\in\prod_{j\geq2}N_j\), the slalom
\(\psi_x(j)=\{x(j)\}\) captures \(x\). Since each \(N_j\) is a
subset of \(\NN\), the product has cardinality at most
\(|\NN^{\NN}|=\mathfrak c\). The family of these singleton slaloms
therefore proves \(\lambda\leq\mathfrak c\).
\end{proof}

The next argument applies one fixed capturing family twice. It combines
restriction to an infinite coordinate set and pullback along injections
\cite[Facts~3.2 and~3.4]{CardonaMejia2023} with finite-block coding and
projection \cite[Lemma~3.14]{CardonaMejia2023}. The block sizes and
width estimates required here are verified directly.

\begin{lemma}\label{lem:alphabet-pullback}
Let \(Y_j\) be nonempty finite sets with \(|Y_j|\leq d_{2j}\) for
\(j\geq2\). For every \(\varepsilon>0\), there is a capturing family on
\((Y_j)_{j\geq2}\) of cardinality at most \(\lambda\) whose coordinate
sets at \(j\) have size at most \(\lfloor2^{\varepsilon j}\rfloor\).
\end{lemma}

\begin{proof}
Fix a capturing family \(\mathcal F\) of cardinality \(\lambda\)
on \((N_i)_{i\geq2}\), with \(|\varphi(i)|\leq2^i\) for every
\(\varphi\in\mathcal F\). Choose injections
\(\iota_j:Y_j\to N_{2j}\) and points \(y_j^0\in Y_j\).
For \(\varphi\in\mathcal F\), put
\[
 R_\varphi(j)=\{y\in Y_j:\iota_j(y)\in \varphi(2j)\},\qquad
 T_\varphi(j)=
 \begin{cases}
 R_\varphi(j),&R_\varphi(j)\ne\varnothing,\\
 \{y_j^0\},&R_\varphi(j)=\varnothing.
 \end{cases}
\]
Then \(T_\varphi(j)\) is nonempty and \(|T_\varphi(j)|\leq2^{2j}\).
For every \(x\in\prod_jY_j\), some \(\varphi\in\mathcal F\) satisfies
\(x(j)\in T_\varphi(j)\) for all but finitely many \(j\). Indeed, define
\(z\in\prod_{i\geq2}N_i\) by \(z(2j)=\iota_j(x(j))\) for
\(j\geq2\), and by \(z(i)=1\) at all remaining coordinates.
A member \(\varphi\) capturing \(z\) has the stated property, because
the source indices \(2j\) are distinct.

For \(n\geq1\), let \(I_n=\{2^n,\ldots,2^{n+1}-1\}\).
These sets partition the coordinates \(j\geq2\). For each
\(\varphi\in\mathcal F\),
\[
 \log_2\left|\prod_{j\in I_n}T_\varphi(j)\right|
 \leq2\sum_{j\in I_n}j
 =3\cdot2^{2n}-2^n<2^{2n+2}
 =\log_2|N_{2n+2}|.
\]
Hence choose an injection
\[
 \eta_{\varphi,n}:\prod_{j\in I_n}T_\varphi(j)\longrightarrow N_{2n+2}.
\]
For \(\varphi,\chi\in\mathcal F\) and \(j\in I_n\), define
\[
 \widetilde\psi_{\varphi,\chi}(j)
 =\left\{v(j):v\in\prod_{k\in I_n}T_\varphi(k),\
                 \eta_{\varphi,n}(v)\in \chi(2n+2)\right\}.
\]
Injectivity of \(\eta_{\varphi,n}\) bounds the number of selected
tuples by \(|\chi(2n+2)|\); projection onto coordinate \(j\)
cannot increase this number. Hence
\[
 |\widetilde\psi_{\varphi,\chi}(j)|\leq|\chi(2n+2)|\leq2^{2n+2}.
\]
Since \((2n+2)/2^n\to0\), choose \(n_0\geq1\) such that
\(2n+2\leq \varepsilon 2^n\) for all \(n\geq n_0\). For such \(n\) and
\(j\in I_n\), the integer \(2^{2n+2}\) is at most \(2^{\varepsilon j}\),
and therefore at most \(\lfloor2^{\varepsilon j}\rfloor\).
Define \(\psi_{\varphi,\chi}(j)=\widetilde\psi_{\varphi,\chi}(j)\) on these blocks,
and \(\psi_{\varphi,\chi}(j)=\{y_j^0\}\) on the finitely many earlier
coordinates. Then \(\psi_{\varphi,\chi}\) has the required width.

To verify capture, fix \(x\in\prod_jY_j\) and choose \(\varphi\) as in
the first paragraph. Replace the finitely many entries of \(x\)
that do not lie in \(T_\varphi(j)\) by members of \(T_\varphi(j)\), obtaining
\(x'\in\prod_j T_\varphi(j)\). Define \(w\in\prod_{i\geq2}N_i\) by
\[
 w(2n+2)=\eta_{\varphi,n}\bigl((x'(j))_{j\in I_n}\bigr)\quad(n\geq1),
\]
and put \(w(i)=1\) at all other coordinates. Choose
\(\chi\in\mathcal F\) capturing \(w\). The indices \(2n+2\) are
distinct, so for all sufficiently large \(n\) the entire tuple
\((x'(j))_{j\in I_n}\) is among the tuples used to define
\(\widetilde\psi_{\varphi,\chi}\). Thus \(x'(j)\in\psi_{\varphi,\chi}(j)\) on
every sufficiently late block. Only finitely many blocks are
exceptional, and each is finite. Since \(x\) and \(x'\) differ
at only finitely many coordinates, \(\psi_{\varphi,\chi}\) captures \(x\).
The family \(\{\psi_{\varphi,\chi}:\varphi,\chi\in\mathcal F\}\) has cardinality
at most \(|\mathcal F|^2=\lambda^2=\lambda\), proving the assertion.
\end{proof}

\begin{lemma}\label{lem:approximating-family}
There is a set \(\mathcal R\subseteq L\) of at most \(\lambda\) points of positive
height such that, whenever \(p\in U\in\sigma(L)\) and
\(\height(p)>0\), some \(b\in\mathcal R\cap U\) satisfies \(b\leq p\).
\end{lemma}

\begin{proof}
For each rational pair \(0<h<k<1\), set \(\varepsilon=(k-h)/2>0\).
For each \(j\geq2\), let
\[
 Y_j=\{R\subseteq N_j:|R|\leq q_k(j)\}.
\]
This finite set is nonempty because it contains \(\varnothing\).
Put \(m_j=q_k(j)\). Since \(0<k<1\), one has
\(0\leq m_j\leq2^j-2\). Counting subsets according to their size gives
\[
 |Y_j|=\sum_{r=0}^{m_j}\binom{d_j}{r}
       \leq(m_j+1)d_j^{m_j}.
\]
Using \(m_j+1\leq2^j\) and \(\log_2d_j=2^j\), we obtain
\[
 \log_2|Y_j|
 \leq\log_2(m_j+1)+m_j2^j
 \leq j+(2^j-2)2^j<2^{2j}.
\]
Thus \(|Y_j|\leq2^{2^{2j}}=d_{2j}\).
For each pair \((h,k)\), fix a capturing family
\(\mathcal F_{h,k}\) supplied by Lemma~\ref{lem:alphabet-pullback},
of cardinality at most \(\lambda\) and with widths
\(w_j=\lfloor2^{\varepsilon j}\rfloor\).
For \(\psi\in\mathcal F_{h,k}\), put
\(H_j^\psi=\bigcup_{R\in\psi(j)}R\). Each member of \(\psi(j)\) has
size at most \(q_k(j)\), so
\begin{equation}\label{eq:approximation-bound}
 |H_j^\psi|\leq w_jq_k(j)
 \leq2^{(1-k+\varepsilon)j}\leq q_h(j)
 \quad\text{for all sufficiently large }j.
\end{equation}
To justify the last inequality including the floor and the subtracted
\(1\), set \(z_j=2^{(1-h)j}\). The identity
\(1-k+\varepsilon=1-h-\varepsilon\) gives \(2^{(1-k+\varepsilon)j}=2^{-\varepsilon j}z_j\).
Choose \(M_0=M_0(h,k)\geq2\) such that \(2^{-\varepsilon j}\leq1/2\) and
\(z_j\geq4\) for \(j\geq M_0\). Then
\[
 2^{-\varepsilon j}z_j\leq z_j/2\leq z_j-2
                 \leq\lfloor z_j\rfloor-1=q_h(j).
\]
This choice of \(M_0\) works for every \(\psi\in\mathcal F_{h,k}\).

For each such \(\psi\), each \(M\geq M_0\), and each finite collection
\((T_j)_{2\leq j<M}\) with \(T_j\subseteq N_j\) and \(c_j(T_j)\geq h\),
form the point \((h,C)\) defined by
\[
 C_j=\begin{cases}
 T_j,&2\leq j<M,\\
 N_j\setminus H_j^\psi,&j\geq M.
 \end{cases}
\]
At the first coordinates, \(c_j(C_j)\geq h\) by the choice of \(T_j\).
At the remaining coordinates,
\(|N_j\setminus C_j|=|H_j^\psi|\leq q_h(j)\), so
\eqref{eq:complement-bound} gives the same inequality.
Thus \((h,C)\in L\) and has positive height.
Let \(\mathcal R\) consist of all points obtained in this way as
\(h,k,\psi,M\) and the finite prefix vary over the choices above.
For a fixed \(M\) there are finitely many choices of \((T_j)_{2\leq j<M}\).
There are countably many cutoffs \(M\) and rational pairs \((h,k)\), and at
most \(\lambda\) choices of \(\psi\) for each pair. Consequently
\[
 |\mathcal R|\leq\aleph_0\cdot\lambda\cdot\aleph_0=\lambda.
\]

It remains to verify the required approximation property.
Let \(p=(t,A)\in U\in\sigma(L)\) with \(t>0\).
By Lemma~\ref{lem:scott-ambient}, \(U\) is \(\tau\)-open.
Choose a relatively open interval \(I\subseteq[0,1]\) and a finite set
\(J\subseteq\{2,3,\ldots\}\) such that
\[
 p\in W=\{(s,C)\in L:s\in I,\ C_j=A_j\ (j\in J)\}
 \subseteq U.
\]
Choose rational numbers \(0<h<k<t\) with \(h\in I\); then \(k<1\).
Since \(c_j(A_j)\geq t>k\), each \(N_j\setminus A_j\) belongs to \(Y_j\).
Choose \(\psi\in\mathcal F_{h,k}\) capturing this sequence, and choose
\(M\geq M_0(h,k)\) greater than every index in \(J\) and every capture
exception. In the construction of \(\mathcal R\), use the prefix \(T_j=A_j\)
for \(j<M\). This prefix is allowed because \(c_j(A_j)\geq t>h\).
For \(j\geq M\), capture gives
\[
 N_j\setminus A_j\in\psi(j),\qquad
 N_j\setminus A_j\subseteq H_j^\psi,
 \qquad N_j\setminus H_j^\psi\subseteq A_j.
\]
Thus the resulting \(b=(h,C)\in\mathcal R\) satisfies \(h<t\) and
\(C_j\subseteq A_j\) for every \(j\), so \(b\leq p\).
It also satisfies \(h\in I\) and \(C_j=A_j\) for \(j\in J\), so
\(b\in W\subseteq U\). This proves the lemma.
\end{proof}

\begin{proof}[Proof of Proposition~\ref{prop:tail-selection}]
Take \(\kappa=\lambda\), where \(\lambda\) is defined in
\eqref{eq:lambda}, and take the reference family \(\mathcal R\) from
Lemma~\ref{lem:approximating-family}.  Lemma~\ref{lem:lambda-bounds}
gives \(\aleph_1\leq\kappa\leq\mathfrak c\), and
Lemma~\ref{lem:approximating-family} is exactly condition
\ref{sel:reference}.

Let \(H\subseteq L_+\) have cardinality less than \(\kappa\).  For
\(x=(t_x,A^x)\in H\), define the slalom
\[
 \psi_x(j)=N_j\setminus A_j^x.
\]
Since \(t_x>0\), equation~\eqref{eq:complement-bound} gives
\(|\psi_x(j)|<2^j\), hence \(|\psi_x(j)|\leq2^j\).  The family
\(\{\psi_x:x\in H\}\) has cardinality less than \(\lambda\), so it is
not capturing by the definition of \(\lambda\).  Choose
\(f\in\prod_{j\geq2}N_j\) captured by none of these slaloms.  For each
\(x\in H\), there are infinitely many \(j\) with
\(f(j)\notin\psi_x(j)\), equivalently \(f(j)\in A_j^x\).  This is
condition \ref{sel:avoidance}.
\end{proof}

\section*{Declaration of generative AI use}

ChatGPT was utilized to explore these constructions, verify intermediate estimates, and assist in organizing and presenting the manuscript. Additionally, further constructions and proofs are available in the supplementary materials. Both the supplementary materials and the Lean formalization code will be made available at \url{https://github.com/ChanYuxu/Recent-Progress-on--Domain-Theory}.

\begingroup
\small
\raggedright
\setlength{\bibsep}{2pt plus .5pt}

\endgroup

\end{document}